\documentclass[11pt]{amsart}
\usepackage{amssymb,amsmath,amsthm}
\newcommand{\shrinkmargins}[1]{
  \addtolength{\textheight}{#1\topmargin}
  \addtolength{\textheight}{#1\topmargin}
  \addtolength{\textwidth}{#1\oddsidemargin}
  \addtolength{\textwidth}{#1\evensidemargin}
  \addtolength{\topmargin}{-#1\topmargin}
  \addtolength{\oddsidemargin}{-#1\oddsidemargin}
  \addtolength{\evensidemargin}{-#1\evensidemargin}
  }
\shrinkmargins{0.5}

\newcommand{\sm}[4]{\left(\begin{smallmatrix}#1&#2\\ #3&#4 \end{smallmatrix} \right)}
\newcommand{\mat}[4]{\left(\begin{matrix}#1&#2\\ #3&#4 \end{matrix} \right)}
\newcommand{\matt}[9] { \left(\begin{matrix} #1&#2&#3\\ #4&#5&#6\\ #7&#8&#9 \end{matrix}  \right)  }
\newtheorem{theorem}{Theorem}
\newtheorem{lemma}[theorem]{Lemma}

\newtheorem{proposition}[theorem]{Proposition}
\newtheorem*{definition}{Definition}

\theoremstyle{remark}

\numberwithin{theorem}{section} \numberwithin{equation}{section}

\newcommand{\Tau}{\tau}
\newcommand{\R}{\mathbb{R}}
\newcommand{\C}{\mathbb{C}}

\newcommand{\Q}{\mathbb{Q}}

\newcommand{\Z}{\mathbb{Z}}
\newcommand{\N}{\mathbb{N}}

\newcommand{\sump}{\sideset{}{'}\sum}

\def\H{\mathbb{H}}
\def\M{\mathbb{M}}
\def\Tau{\mathcal{T}}

\begin{document}
\title{Hecke Duality Relations of Jacobi forms}
\author{Kathrin Bringmann}
\address{School of Mathematics\\University of Minnesota\\ Minneapolis, MN 55455 \\U.S.A.}
\email{bringman@math.umn.edu}  
\author{Bernhard Heim}
\footnote{The main part of this research was carried out at the \textit{Mathematisches Forschungsinstitut Oberwolfach}  during  the \textit{Research in Pairs} program from  june 3-16 2007}
\address{Max-Planck Institut  f\"ur Mathematik, Vivaitsgasse 7, 53111 Bonn, Germany}
\email{heim@mpim-bonn.mpg.de}

\date{\today}
\begin{abstract}
In this paper we introduce a new  subspace  of Jacobi forms of higher degree  via 
certain  relations among    Fourier coefficients. 
We prove that this space can also be characterized by  duality properties of   certain  distinguished 
embedded Hecke operators.   
We then show that this space  is Hecke invariant with respect to all good  Hecke operators.
 As explicit examples we give    Eisenstein series.   
 Conversely we show the existence of forms that are  not contained in this space.  
\end{abstract}
\maketitle
\section{Introduction and statement of results}  
In this paper we prove the existence of a non-trivial Hecke invariant proper subspace of the space of 
Jacobi forms on $\H_2 \times \C^2$ which satisfies  Hecke duality relations. 

Thirty years ago Saito und Kurokawa \cite{Ku78} conjectured  the existence of a distinguished subspace 
of $M_k^2$, the vector space of Siegel modular forms of degree $2$ and weight $k\in \N$. From the degeneration 
of the spinor   $L$-function of Hecke eigenforms
they conjectured that there exists a correspondence to a  space of elliptic modular forms.
At the same time,  Maass studied the Fourier coefficients of Eisenstein series of degree $2$ 
and discovered interesting relations
among them. He introduced the Spezialschar \cite{Ma79I, Ma79II,Ma79III}, which is 
a subspace of Siegel modular forms of degree $2$  that 
is defined via  certain relations among Fourier coefficients.  
A Siegel modular form $F \in M_k^2$ is in the \textit{Maass Spezialschar} if for all
 positive definite half-integral $2 \times 2$ matrices $T$ the Fourier coefficients 
$A(T)$ of $F$ satisfy the relation
\begin{equation} \label{Maassrelation}
A\big([n,r,m]\big) = \sum_{d \vert (n,r,m)} d^{k-1} A\left(\left[\frac{nm}{d^2},\frac{r}{d},1\right]\right),
\end{equation}
where   we identify $T = \sm{n}{r/2}{r/2}{m}$ with the quadratic form $[n,r,m]$.  Relation
(\ref{Maassrelation}) is   equivalent to the fact that  $A(T)$ only depends on the discriminant of $T$ and 
the greatest common divisor of its entries. 
Explicit examples for elements in the Spezialschar are given by   Eisenstein series.      
Moreover the   Maass Spezialschar is Hecke invariant and     provides an  example
of forms that do not satisfy   an analogue of  the Ramanujan-Petersson  conjecture in higher dimensions \cite{Za80}.

At the end of the last century Duke and Imamoglu generalized   the Saito-Kurokawa conjecture to higher degree.  Let  $M_k^n$ be the space of Siegel modular forms of weight $k$ and degree $n \in \N$.
Duke and Imamoglu conjectured that if $n+k \equiv 0 \pmod{2}$, then  there exists   a  Hecke invariant isomorphism  between the space of the elliptic forms of weight $2k$ and a subspace of $M_{k+n}^{2n}$. Note that $M_{k+n}^{2n}$  has  even degree.
In 2001 this conjecture was proven by Ikeda \cite{Ik01} and subsequently    Kohnen and Kojima \cite{KK05} gave a linearization of this lift.   
Recently the second author characterized these lift via Hecke operators   \cite{He06}.  

It would be  interesting to construct 
lifts and formulas  also in the case of Siegel modular forms of odd degree.  
Here we make a first  step towards this goal and investigate 
Jacobi forms of degree $2$ since these 
 arise for example as Fourier Jacobi coefficients of Siegel modular forms of degree $3$.  
Let $\Phi \in J_{k,m}^2$   (the vector space of Jacobi forms of weight $k$, degree $2$, and index $m$)
with Fourier coefficients $C(N,R)$.  
We use the parametrization  
$ N = \left( \begin{smallmatrix}
n_{11} & \frac{n_{12}}{2} \\
\frac{n_{12}}{2} & n_{22}
\end{smallmatrix} \right)$
and $R = (r_1,r_2)$, and   define the invariants:
\begin{equation}\label{invariants}
D_1 := -4 \cdot \det \left(\begin{matrix} n_{11}&\frac{r_{1}}{2} \\  \frac{r_{1}}{2}  & m \end{matrix} \right),\qquad
D_2 :=  -4 \cdot\det \left(\begin{matrix} n_{22}&\frac{r_{2}}{2} \\  \frac{r_{2}}{2}  & m \end{matrix} \right),
\qquad D :=  -4 \cdot \det \left(\begin{matrix} \frac{n_{12}}{2} &\frac{ r_{1}}{2} \\  \frac{r_{2}}{2}  & m \end{matrix} \right).
\end{equation}  
It will  turn out    that $C(N,R)$ only 
depends upon $D_1,D_2, D$, and $r_1$ and $r_2$ modulo $2m$.  Therefore the following is well-defined: 
\begin{equation}\label{coefficientallg}
C_{r_1,r_2}(D_1,D_2,D):= 
C (N,R),
\end{equation} 
where $r_1$ and $r_2$ are defined modulo $2m$.   
Moreover if $m$ is either $1$ or a prime, then   $C(N,R)$ only depends upon $D_1,D_2$, and $D$. In this case we set  
\begin{equation}\label{coefficient}
C(D_1,D_2,D):= 
C (N,R).
\end{equation}   
Next define the space  $\mathbb{E}_{k,m}$ of ``distinguished" Jacobi forms.  
\begin{definition}[Hecke  duality relation] 
Let $\Phi \in J_{k,m}^2$ with Fourier coefficients $C_{r_1,r_2}(D_1,D_2,D)$ with  
$D_1,D_2,D,r_1,r_2 \in \Z$.
 Let $p$ be a prime with  $(p,2m)=1$.
Then $\Phi$ is an element  of $\mathbb{E}_{k,m}^{(p)}$ if the following relation is satisfied: 						 
\begin{eqnarray} \label{Hecke_Duality_Relations}
\Big( \chi_{{ D_1}}(p) -\chi_{D_2}(p) \Big) \,\, C_{r_{1},r_{2}} \left( D_1,  D_2,D  \right)		&= & \nonumber \\ & &
\!\!\!\!\!\!\!\!\!\!\!\!\!\!\!\!\!\!\!\!
\!\!\!\!\!\!\!\!\!\!\!\!\!\!\!\!\!\!\!\!
\!\!\!\!\!\!\!\!\!\!\!\!\!\!\!\!\!\!\!\!
\!\!\!\!\!\!\!\!\!\!\!\!\!\!\!\!\!\!\!\!
p^{2-k}\, \left( C_{r_{1},pr_2} \left( D_1,D_2 p^2,  D p \right) - 	C_{p r_{1},r_{2}} \left( D_1 p^2,  D_2,D p  \right) \right) \\
& & 
\!\!\!\!\!\!\!\!\!\!\!\!\!\!\!\!\!\!\!\!
\!\!\!\!\!\!\!\!\!\!\!\!\!\!\!\!\!\!\!\!
\!\!\!\!\!\!\!\!\!\!\!\!\!\!\!\!\!\!\!\!
\!\!\!\!\!\!\!\!\!\!\!\!\!\!\!\!\!\!\!\!
+
p^{k-1} \left( C_{r_{1}, \bar p r_2} \left(D_1,  \frac{D_2}{p^2},  \frac{D}{p}  \right)
-
C_{\bar p r_{1},r_{2}} \left( \frac{D_1}{p^2}, D_2, \frac{D}{p} \right) \right).\nonumber
\end{eqnarray}
							Here $p \bar p \equiv 1 \pmod{2m}$ and 
							$\chi_{*} := \left( \frac{*}{p} \right)$. 
							Furthermore we define  
							  \begin{equation} \label{eulen}
							  \mathbb{E}_{k,m} :=   \bigcap_{\substack{ p \text{ prime}\\
							  (p,2m)=1}} \mathbb{E}_{k,m}^{(p)}. 
							  \end{equation}
					\end{definition} 
					Since we   show  that this relation is equivalent to
a property involving   Hecke Jacobi operators,
  we  refer to it as    \textit{Hecke duality relation}.   
For $m=1$ it is known by work of Ibukiyama \cite{Ib} that the space $J_{k,1}^2$ is isomorphic to a space of Siegel modular forms of degree $2$ and half-integral weight.  
In this case he   conjectured the existence of a certain distinguished  subspace which seems to be different from the space considered here. It would be interesting to determine the connection between those.  
\begin{theorem} \label{DualityTheorem}
Assume that  $\Phi \in J_{k,m}^2$ and that $p$ is a prime with   $(p,2m) =1$. Then the following two 
conditions are equivalent:
\begin{enumerate}
\item The function
$\Phi$ is an element of  $\mathbb{E}_{k,m}^{(p)}$.
\item  We have  
$\Phi \vert \left(  T^J(p)^\uparrow - T^J(p)^\downarrow\right) = 0$.  
\end{enumerate}
Here $T^J(p)^{\uparrow}$ and $T^J(p)^{\downarrow}$  are two canonical Hecke Jacobi operators obtained by  
embedding the classical Hecke Jacobi  operator $T^J(p)$ in two different ways (see Section \ref{heckejacobi}). 
\end{theorem} 
The
spaces $\mathbb{E}_{k,m}^{(p)}$  are invariant with respect to all ``good" Hecke operators.  
\begin{theorem} \label{HeckinvTheorem}
Assume that  $p$ and  $q$ are distinct primes with $(pq,2m)=1$. 
Let  $\mathcal{H}_p^{J,2}$ be  the 
local Hecke Jacobi algebra of degree $2$ and define $\mathcal{H}^{J,2} := \otimes_{(p,2mq)=1} \mathcal{H}_p^{J,2}$. 
If $m$ is either $1$ or a prime,  then  $\mathbb{E}_{k,m}^{(q)}$ is invariant with respect to   $\mathcal{H}^{J,2} $. 
\end{theorem} 
Next we prove that Jacobi Eisenstein series $E_{k,m}^{J,2}$  (see Section  \ref{EisenSection}) are 
  contained in $\mathbb{E}_{k,m}$. 
  As a by-product we show a  new decomposition of Eisenstein series. 
\begin{theorem} \label{EisenTheorem}  
Assume that $m$ is square-free. Then   we have  for all primes $p$
\begin{equation} \label{HeckeEisen}
E_{k,m}^{J,2}| 
\left(T^J(p)^{\uparrow} -T^J(p)^{\downarrow} \right)=0.
\end{equation}
If $m$ is arbitrary, then (\ref{HeckeEisen}) is also true if $(p,m)=1$. 
\end{theorem}
To show that  
$\mathbb{E}_{k,m}$ is not in general  equal to the whole space  of Jacobi forms, 
we give in the last section explicit examples which do not satisfy
the Hecke duality relations. In particular  if   $k \geq 10$ is even, 
 then we have $\mathbb{E}_{k,1} \subsetneqq J_{k,1}^2$.
\section{Basic facts about Automorphic forms}\label{automorphicforms}
In this section we recall some basic facts about automorphic forms with
 respect to the symplectic and  to the Jacobi group.  
 Throughout we let $R$ be a commutative ring.  
The symplectic  group $Sp_n(R)$ acts on the Siegel upper half-space  $\H_n$    by 
\begin{eqnarray*}
g \circ  \Tau := 
(A\Tau+B)(C\Tau+D)^{-1}.
\end{eqnarray*}
If $g= \left(\begin{smallmatrix}A&B\\C&D \end{smallmatrix} \right) \in Sp_n(R)$, $k \in \N_0$, 
 and $\Phi$ is a complex-valued function on $\H_n$, then define 
\begin{eqnarray*}
F|_k g (\Tau) := \left(\det J(g,\Tau)\right)^{-k} F(  g \circ \Tau ),
\end{eqnarray*}
where  $J(g,\Tau) := \left( C \Tau + D \right)$.  
We let   $M_k^n$ be the vector space of Siegel modular forms of weight $k$ and degree $n$ with respect 
to the Siegel modular group $\Gamma_n := Sp_n(\Z)$, i.e., the space of holomorphic functions 
$F:\H_n \to \C$ that satisfy $F|_kg=F$ for all $g \in \Gamma_n$  and that have a Fourier expansion  
\begin{equation*}
F(\Tau) = \sum_S \,\, A(S) \,\, e^{2 \pi i \text{tr} \, ( S \Tau)},
\end{equation*}
where $S$ runs through the set of half-integral semi-definite matrices.

We next consider Jacobi forms. 
The Heisenberg group   of degree $n$
$$
H_n(R):= \{ 
\left(\lambda,\mu,\kappa \right)
\vert \, \lambda, \mu \in R^n \text{ and } \kappa \in R\}
$$ 
has the   group law  
\begin{equation*}
\left(\lambda_1,\mu_1,\kappa_1 \right) +\left(\lambda_2,\mu_2,\kappa_2 \right) = \left(\lambda_1 + \lambda_2,\mu_1+ \mu_2,
\kappa_1 + \kappa_2 + \lambda_1 \mu_2^t - \mu_1 \lambda_2^t
\right).
\end{equation*}
Define the Jacobi group $G_n^J(R):=  Sp_n(R) \ltimes H_n(R)$. 
This  group can be viewed as a subgroup of $Sp_{n+1}(R)$ via the   embedding
$\widehat{\phantom{x}}: G_n^J(R) \longrightarrow Sp_{n+1}(R),$ where $\gamma = (g,h) = \left( \sm{A}{B}{C}{D}, ( \lambda,\mu,\kappa) \right)$
maps to
\begin{eqnarray*}
\widehat{\gamma} :=
\left( \begin{matrix}A&0&B&\mu'\\ \lambda&1&\mu&\kappa\\ C&0&D&-\lambda'\\ 0&0&0&1 \end{matrix}  \right).
\end{eqnarray*}
Here $\lambda',\mu'$ are uniquely determined. 
Moreover we have $\widehat{\gamma} = \widehat{g} \cdot \widehat{h}$.  
The group  $G_n^J(R)$ acts on $\H_n^J$ via 
\begin{equation*}
\gamma \circ ( \Tau,Z):= \left( g\circ \Tau, \left(Z \,  + \lambda \Tau + \mu \right) J^{-1}(g,\Tau)
 \right).
\end{equation*}
We   define the cocycle
$J_{k,m}$ by
\begin{equation*}
J_{k,m}\big((g,h),\Tau,Z\big):= J_{k,m}\big(g,h \circ(\Tau, Z)\big) \cdot J_{k,m}\big(h,(\Tau,Z)\big),
\end{equation*}
where
\begin{eqnarray*}
J_{k,m}\big(g,(\Tau, Z)\big) &:=&    \text{det}\left( J(g,\Tau \right) \,\, 
e\left(m \,\text{tr}    \left( \left( J^{-1}(g,\Tau)C\right)[Z^t]\right)\right),       
\\
J_{k,m}\left(h,(\Tau,Z)\right) 
&:=&
 e\left(-m \, \text{tr} \left( \Tau[\lambda^t] + 2 \lambda Z^t + \kappa + \mu \lambda^t \right)\right)
.
\end{eqnarray*}
Here  $ A [B]:= B^t A B$ for  matrices of suitable sizes and $e(x):=  e^{ 2 \pi i x}$ for  $x \in \C$.
For  all $\gamma_2,\gamma_2 \in G_n^J(\R)$, 
the  cocyle $J_{k,m}$ has  the property
\begin{equation*}
J_{k,m} \big( \gamma_1 \gamma_2, (\Tau,Z) \big) = 
J_{k,m} \big( \gamma_1, \gamma_2 \circ(\Tau,Z) \big)\,\,
J_{k,m} \big( \gamma_2, (\Tau,Z) \big).
\end{equation*}
Define the Petersson slash operator $\vert_{k,m}$ for complex-valued functions 
$F$ on $\H_n^J$ and $\gamma \in G_n^J(\R)$ by 
\begin{equation*}
\Phi \vert_{k,m} \gamma (\Tau,Z) :=
J^{-1}_{k,m}\left(\gamma,(\Tau,Z) \right)\,\, \Phi \left( \gamma \circ (\Tau,Z) \right).
\end{equation*} 
For positive integers   $k,m$, and $n$  we let  $J_{k,m}^n$ be the space of Jacobi forms of degree $n$, weight $k$, and index $m$, i.e., the space of complex-valued functions   $\Phi$ on $\H_n^J$ 
that satisfy  $\Phi\vert_{k,m} \gamma =\Phi$ for $\gamma \in 
\Gamma_n^J := G_n^J(\Z)$ and that have a Fourier expansion   of the form
\begin{equation*}
\Phi \left( \Tau, Z \right) 
= 
\sum_{N,R} 
C(N,R) \,\, e^{ 2 \pi i \left( \text{tr} (N \Tau ) + R Z^t \right)}.
\end{equation*}
Here the sum runs over all   $N$ and $R$ such that $\left(\begin{smallmatrix} N & R^t \\ R & m \end{smallmatrix}\right)$
is a half-integral semi-definite matrix of size $n+1$.   
Examples of  Jacobi forms are given by  Fourier Jacobi coefficients of Siegel modular forms. 
If we   write for $n>1$ an element  of $\H_n$ as 
 $\left( \begin{smallmatrix} \Tau & Z^t \\ Z & \omega_0 \end{smallmatrix}\right)$ with 
  $\omega_0 \in \H$, then $F \in M_k^n$  has a Fourier Jacobi expansion of the form 
\begin{equation*}
F \left( \begin{matrix} \Tau & Z^t \\ Z & \omega_0 \end{matrix}\right) = 
\sum_{m=0}^{\infty} \Phi_m^F \left( \Tau,Z\right) \,\, e(m \,\omega_0),
\end{equation*}
where   $\Phi_m^F \in J_{k,m}^{n-1}$.  
We drop the index $n=1$ to simplify notation.  
For  $\Phi \in J_{k,m}^n$ we define
\begin{equation*}
\widehat{\Phi} 
\left( \begin{matrix} \Tau & Z^t \\ Z & \omega_0 \end{matrix}\right) := 
\Phi \left( \Tau, Z \right) \,\, e(m \, \omega_0),
\end{equation*} 
which has the property 
\begin{equation}\label{transform1}
\Phi \vert_{k,m} \gamma = e(-m \, \omega_{0}) \,\, 
 \widehat{\Phi}\vert_k \widehat{\gamma}.
\end{equation} 
We next embed 
$Sp_n \times Sp_n$ into $Sp_{2n}$   by  
\begin{equation*}
Sp_n \times Sp_n \to  Sp_{2n},\qquad
\left(\begin{smallmatrix}A_1&B_1\\ C_1&D_1 \end{smallmatrix} \right)  
\times \left(\begin{smallmatrix} A_2&B_2\\
C_2& D_2 \end{smallmatrix} \right)
\mapsto \left(\begin{smallmatrix}  
A_1&0 &B_1&0 
\\ 0& A_2 &0& B_2\\ 
C_1&0&D_1&0\\
0&C_2&0&D_2
 \end{smallmatrix} \right)
\end{equation*}
and identify this image with $Sp_n \times Sp_n$. 
Further we embed   $G_n^J \times G_n^J$ into 
$G_{2n}^J$ via  \begin{eqnarray*}
\Big( 
( \lambda_1,\mu_1, \kappa_1),g_1\Big) \times 
\left(  
\left( \lambda_2, \mu_2, \kappa_2\right), g_2   \right)
\mapsto \left(    \left(  
(\lambda_1, \lambda_2),  (\mu_1, \mu_2), (\kappa_1 + \kappa_2) , g_1 \times g_2  ) 
\right)\right).
\end{eqnarray*} 
In the following we use the symbols $\gamma^\uparrow$ and 
$\widetilde{\gamma}^\downarrow$ to indicate the embeddings of $\gamma \times I_2$ and $I_2 \times \widetilde{\gamma}$. 
\section{Hecke Theory and the lifting operator} \label{HeckeSection}
\subsection{The symplectic Hecke algebra} 
Let us first consider the symplectic Hecke algebra  $\mathcal{H}^n$  of  the Hecke pair 
 $\left( \Gamma_n, Sp_n( \Q) \right)$ \cite{An87,Sh71}  which   
  decomposes  as   $\mathcal{H}^n = \otimes_p \, \mathcal{H}_p^n$. 
Here the local Hecke algebra $\mathcal{H}_p^n$ is generated by
$\left( T_i^{(n)}(p) \right)_{(0 \leq i \leq n)}$, where
$$T_i^{(n)}(p) := 
p^{-1}
\, \Gamma_n \, 
\text{diag} \left( 1,\ldots,1,\underbrace{p,\ldots,p}_{i}; p^2,\ldots,p^2,\underbrace{p \ldots,p}_{i}\right) \Gamma_n.$$ 
We need  an explicit left coset decomposition of the generators of the Hecke operators for  $n=1$ and $n=2$.   
Since  $T_n^{(n)}(p) = I_{2n}$, we  can omit the case $i=n$.   

If $n=1$ we    can choose  as a
$\Gamma$-left coset decomposition of $p \,T_0^{(1)}(p)$:
\begin{equation}\label{einsp2}
\Gamma \, N_1  
+
\sum_a 
\Gamma \, N_2(a)  
+
\sum_b
\Gamma \, N_3(b),  
\end{equation}
where $a$ runs through $\left( \Z / p\Z\right)^{\ast}$, $b$ through  $\left( \Z / p^2\Z\right)$, and where  $N_1:=\sm{p^2}{0}{0}{1}$,    $N_2(a):= \sm{p}{a}{0}{p}$, and $N_3(b):= \sm{1}{b}{0}{p^2}$. 

If $n=2$ the generators of the Hecke algebra  are given by: 
\begin{equation*} 
T_0^{(2)} (p)  =  
p^{-1} \Gamma_2 \,
\left( \begin{matrix}1&0&0 &0\\ 0&1&0&0\\ 0&0&p^2&0\\ 0&0&0&p^2 \end{matrix}  \right)\,
\Gamma_2,   \qquad 
T_1^{(2)} (p)  =  
p^{-1} \Gamma_2 \,
\left( \begin{matrix}1&0&0 &0\\ 0&p&0&0\\ 0&0&p^2&0\\ 0&0&0&p \end{matrix}  \right)\,
\Gamma_2,\qquad T_2^{(2)} (p) = \Gamma_2.
\end{equation*}  
Representatives can be choosen  of  the form 
$
p^{-1} \left( \begin{smallmatrix} A & B \\ 0 & D \end{smallmatrix} \right) \in Sp_2 ( \Q),
$ 
with $D = \left(\begin{smallmatrix} * & * \\ 0 & * \end{smallmatrix}\right)$,
$A = p^2 (D^t)^{-1}$, and $D^t \, B = B^t \, D$ (see \cite{An87}). 
Then   $D$ runs through all $\Gamma$-left cosets of
$
\Gamma \backslash \left( \bigcup_{j=1}^6 \Gamma {\mathcal{D}}_j \Gamma \right)
$
with \begin{equation*}
\mathcal{D}_1 = I_2,\quad
\mathcal{D}_2 = \left(\begin{matrix}  1 & 0 \\ 0 & p\end{matrix} \right),\quad
\mathcal{D}_3 = pI_2,\quad
\mathcal{D}_4 =   \left(\begin{matrix}  1 & 0 \\ 0 & p^2 \end{matrix} \right)     ,\quad
\mathcal{D}_5 = p \left(\begin{matrix}  1 & 0 \\ 0 & p\end{matrix} \right),\quad
\mathcal{D}_6 = p^2I_2.
\end{equation*} 
Each of 
the    double cosets related to $\mathcal{D}_1, \mathcal{D}_3$, and $ \mathcal{D}_6$
decompose into one left coset.
From the  decomposition of $\Gamma \backslash \Gamma \mathcal{D}_4 \Gamma$  in (\ref{einsp2}) and the identity
 $p\mathcal{D}_2 = \mathcal{D}_5$, it is sufficient to consider
\begin{equation}
\Gamma \, 
\left(\begin{matrix}  1 & 0 \\ 0 & p\end{matrix} \right)  \,\Gamma  = 
\Gamma \,
\left(\begin{matrix}  p & 0 \\ 0 & 1\end{matrix} \right) + \sum_{ a \!\!\!\!    \pmod{p}}  \Gamma \, 
\left(\begin{matrix}  1 & a \\ 0 & p\end{matrix} \right).
\end{equation}
Next we  calculate the corresponding representatives $M \in Sp_2(\Q)$. 
If a representative $D \in \text{Mat}(2,\Z)$ is fixed,
then $B \in \text{Mat}(2,\Z)$ runs through a set of representatives modulo $D$, i.e.,
$B$ satisfies $D^t \, B = B^t \, D$ and the congruence relation $\sim$. Here
$ B_1 \sim B_2$ if and only if $(B_1 - B_2) D^{-1} \in Mat(2,\Z)$.
Using the algorithm given in \cite{HW} we obtain  
 the following representatives:
 \begin{displaymath}
 \begin{array}{ll}
 M_1(x,y,z):=\left(
\begin{matrix}
p^{-1}&0&\frac{y}{p}&\frac{x}{p}\\
0&p^{-1}&\frac{x}{p}&\frac{z}{p}\\
0&0&p&0\\
0&0&0&p
\end{matrix}
  \right)_{  x,y,z  \!\!\!\!\!  \pmod{p^2}}
  & 
  M_2(s,x):=\left(
\begin{matrix}
p^{-1}&0&\frac{s}{p}&\frac{x}{p}\\
0&1&x&0\\
0&0&p&0\\
0&0&0&1
\end{matrix}
  \right)_{  \substack{ s  \!\!\!\!\!  \pmod{p^2}\\ x  \!\!\!\!\!  \pmod{p}}}\\
M_3:=\left(
\begin{matrix}
p&0&0&0\\
0&p&0&0\\
0&0&p^{-1}&0\\
0&0&0&p^{-1}
\end{matrix}
  \right) &
M_4(a):=\left(
\begin{matrix}
p^{-1}&0&\frac{a}{p}&0\\
0&p&0&0\\
0&0&p&0\\
0&0&0&p^{-1}
\end{matrix}
  \right)_{a \!\!\!\!\!  \pmod{p^2}}\\ 
M_5:=\left(
\begin{matrix}
1&0&0&0\\
0&p&0&0\\
0&0&1&0\\
0&0&0&p^{-1}
\end{matrix}
  \right) &
  M_{6}(a,b):=\left(
\begin{matrix}
p&0&0&0\\
-a&1&0&\frac{b}{p}\\
0&0&p^{-1}&\frac{a}{p}\\
0&0&0&1
\end{matrix}
  \right)_{  \substack{a \!\!\!\!\! \pmod{p} \\b \!\!\!\!\! \pmod p*  }}\\ 
M_7(a,b,c):=\left(
\begin{matrix}
1&0&0&\frac{c}{p}\\
-\frac{a}{p}&1&\frac{c}{p}&\frac{b}{p}\\
0&0&1&\frac{a}{p}\\
0&0&0&1
\end{matrix}
  \right)_{  \substack{ a, b,c  \!\!\!\!\! \pmod{p}\\ a \not\equiv 0  \, (mod \, p)}}&
M_8(a,x):= \left(
\begin{matrix}
p&0&0&0\\
-\frac{a}{p}&\frac{1}{p}&0&\frac{x}{p}\\
0&0&p^{-1}&\frac{a}{p}\\
0&0&0&p
\end{matrix}
  \right)_{ a, x   \!\!\!\!\! \pmod{p^2}}
  \end{array}
  \end{displaymath}
  \begin{displaymath}
  \begin{array}{ll}
M_9(a,x,z,\kappa)
:= \left(
\begin{matrix}
1&0&\frac{x}{p}&z\\
-\frac{a}{p}&\frac{1}{p}&\frac{z}{p}&\frac{\kappa}{p}\\
0&0&1&a\\
0&0&0&p
\end{matrix}
  \right)_{  \substack{ a, z   \!\!\!\!\! \pmod{p}\\ \kappa  \!\!\!\!\! \pmod{p^2}}}&
M_{10}:=I_4\\ 
M_{11}(a):=\left(
\begin{matrix} 
p&0&0&0\\
-a&1&0&0\\
0&0&p^{-1}&\frac{a}{p}\\
0&0&0&1
\end{matrix}
  \right)_{  a \!\!\!\!\! \pmod{p}}& 
M_{12}(a,x,z,\kappa)
:= \left(
\begin{matrix}
1&0&\frac{x}{p}&\frac{xa+zp}{p}\\
-\frac{a}{p}&\frac{1}{p}&\frac{z}{p}&\frac{\kappa}{p}\\
0&0&1&a\\
0&0&0&p
\end{matrix}
  \right)_{ \substack{ \substack{ a, z   \!\!\!\!\! \pmod{p}\\ \kappa  \!\!\!\!\! \pmod{p^2}}  \\ x  \!\!\!\!\! \pmod{p}^* }}.
\end{array}
\end{displaymath}    
\subsection{Hecke Jacobi operators}\label{heckejacobi} 
In the setting of Jacobi forms complications arise since the Jacobi group is   not reductive.
It is well known that the related Hecke Jacobi algebra in not commutative 
and does not  decompose into   local Hecke algebras. 
For our purpose it will   be sufficient to consider double cosets attached to the symplectic part
of the Jacobi group $G_n^J(\Q)$ which includes   the Hecke operators   introduced in \cite{EZ85,Mu89}.

For $l \in \N$ we define
$ \mathbb{X}(l):= \Gamma^J \left(\begin{smallmatrix} l & 0 \\ 
0 & l^{-1} \end{smallmatrix}\right)\Gamma^J$. 
We have $ \mathbb{X}(l_1 l_2)= \mathbb{X}(l_1) \cdot \mathbb{X}(l_2)$ for all positive coprime integers $l_1$ and $l_2$.
Let $\mathcal{H}^J$ be the Hecke Jacobi algebra generated by $\{ \mathbb{X}(l) \vert \, l \in \N\}$ over $\Q$ and
$\mathcal{H}_p^J$ the local Hecke Jacobi algebra generated by $\{ \mathbb{X}(p^n) \vert \, n \in \N_0\}$ over $\Q$.
Then we have $\mathcal{H}^J = \otimes_p \mathcal{H}_p^J.$   
In the following we identify  the Hecke Jacobi algebra
and the related Hecke Jacobi operators.  

For $\Phi \in J_{k,m}$ define the Hecke operator $T^J(\ell)$ as 
\begin{equation}  \label{Heckeop}
\Phi| T^J(\ell):= \ell^{k-4}
\sum_{ \substack{\substack{M \in Mat(2,\Z)  \\ \det (M) = \ell^2  } \\ \text{gcd}(M)= \Box}}
\quad
\sum_{(\lambda, \mu) \in \Z^2 \slash \ell \Z^2 }
\Phi|_{k,m } \left( \frac{1}{\ell} M (\lambda,\mu) \right).
\end{equation} 
Then  $T^J(p)$ is related to  $\mathbb{X}(p)$ by  
$\mathbb{X}(p) = p^{3-k} T^J(p)$.    
Using   (\ref{transform1}),  we can rewrite (\ref{Heckeop}) as
\begin{equation*}
e( -m \,\, \omega_{0}) \,\, 
\ell^{k-4} \!\!  \!\! \!\! \!\!
\sum_{ \substack{\substack{M = \left(\begin{smallmatrix} a & b \\ c & d \end{smallmatrix} \right) 
\in Mat(2,\Z)  \\ \det (M) = \ell^2  } \\ \text{gcd}(M)= \Box}}\,\,\,
\sum_{(\lambda, \mu) \in \Z^2 \slash \ell \Z^2 }
\hat{\Phi }|_k  \left(   \frac{1}{\ell}  \left(\begin{smallmatrix} a&0&b&0\\ 0&1&0&0\\c&0&d&0\\ 0&0&0&1\end{smallmatrix} \right)
\left(\begin{smallmatrix} 1&0&0&\mu\\ \lambda&1&\mu&\kappa\\ 0&0&1&-\lambda\\ 0&0&0&1\end{smallmatrix} \right)   \right).
\end{equation*}
For   $\Phi \in J_{k,m}^{2}$ we introduce the operators $T^J(l)^\uparrow$ and $T^J(l)^\downarrow$:
\begin{eqnarray*}
\Phi \vert T^J(l)^\uparrow &:= &
\frac{e( -m \,\, \omega_{0})}{\ell^{ 4-k}}  \,
\sum_{ \substack{\substack{M = \left(\begin{smallmatrix} a & b \\ c & d \end{smallmatrix} \right) 
\in Mat(2,\Z)  \\ \det (M) = \ell^2  } \\ \text{gcd}(M)= \Box}}  
\sum_{(\lambda, \mu) \in \Z^2 \slash \ell \Z^2 }
\hat{\Phi }|_k 
\left(   \frac{1}{\ell} 
\left(\begin{smallmatrix}
a & 0 & 0 & b & 0 & 0 \\
0 & 1 & 0 & 0 & 0 & 0 \\
0 & 0 & 1 & 0 & 0 & 0 \\
c & 0 & 0 & d & 0 & 0 \\
0 & 0 & 0 & 0 & 1 & 0 \\
0 & 0 & 0 & 0 & 0 & 1 
\end{smallmatrix}\right)
\left(\begin{smallmatrix}
1 & 0 & 0 & 0 & 0 & \mu \\
0 & 1 & 0 & 0 & 0 & 0 \\
\lambda & 0 & 1 & \mu & 0 & 0 \\
0 & 0 & 0 & 1 & 0 & - \lambda \\
0 & 0 & 0 & 0 & 1 & 0 \\
0 & 0 & 0 & 0 & 0 & 1 
\end{smallmatrix}\right)
\right)\label{oben},
\end{eqnarray*} 
\begin{eqnarray*}
\Phi \vert T^J(l)^\downarrow &:= &
\frac{e( -m \,\, \omega_{0})}{\ell^{4-k}} \,  
\sum_{ \substack{\substack{M = \left(\begin{smallmatrix} a & b \\ c & d \end{smallmatrix} \right) 
\in Mat(2.\Z)  \\ \det (M) = \ell^2  } \\ \text{gcd}(M)= \Box}}  
\sum_{(\lambda, \mu) \in \Z^2 \slash \ell \Z^2 }
\hat{\Phi }|_k 
\left(   \frac{1}{\ell} 
\left(\begin{smallmatrix}
1 & 0 & 0 & 0 & 0 & 0 \\
0 & a & 0 & 0 & b & 0 \\
0 & 0 & 1 & 0 & 0 & 0 \\
0 & 0 & 0 & 1 & 0 & 0 \\
0 & c & 0 & 0 & d & 0 \\
0 & 0 & 0 & 0 & 0 & 1 
\end{smallmatrix}\right)
\left(\begin{smallmatrix}
1 & 0 & 0 & 0 & 0 & 0 \\
0 & 1 & 0 & 0 & 0 & \mu \\
0 & \lambda & 1 & 0 & \mu & 0 \\
0 & 0 & 0 & 1 & 0 & 0 \\
0 & 0 & 0 & 0 & 1 & - \lambda \\
0 & 0 & 0 & 0 & 0 & 1 
\end{smallmatrix}\right)
\right).\label{unten}
\end{eqnarray*}  
For $S \in Sp_n(\Q)$ and $\Phi \in J_{k,m}^n$, we define the Hecke Jacobi operators $T_n^J(S)$:
\begin{equation*}
\Phi|T_n^J(S) := \sum_{g \in \Gamma_n \backslash \Gamma_n S \Gamma_n}  \sum_{h \in M(l)} 
\Phi|_{k} gh,
\end{equation*}
where $l$ is the smallest integer such that $lS \in Mat(2n,\Z)$ and $M(l)$ is the set of all $M(\lambda,\mu)$ with  $\lambda,\mu \in \Z^n \slash l \Z^n$. 
We note that $T_1^J \sm{1}{0}{0}{1}$ and $T_1^J\sm{p}{0}{0}{1/p}$ generate $\mathcal{H}_p^J$.
Let further $\mathcal{H}_p^{J,2}$ be the Hecke Jacobi algebra generated by 
$(T^J(S_i))_i$, where
\begin{equation} \label{inv}
 S_1:= I_4, \qquad 
 S_2:= \left( \begin{smallmatrix} 1/p&0&0&0\\0&1/p&0&0\\ 0&0&p&0\\ 0&0&0&p\end{smallmatrix}\right), \qquad 
 S_3:= \left( \begin{smallmatrix} 1/p&0&0&0\\0&1&0&0\\ 0&0&p&0\\ 0&0&0&1\end{smallmatrix}\right).
 \end{equation}
\section{Proof of Theorem \ref{DualityTheorem}} \label{DualitySection} 
Throughout, we let  $\Phi \in  J_{k,m}^2$ with Fourier coefficients $C(N,R)$ and write 
$N = \left( \begin{smallmatrix}
n_{11} & \frac{n_{12}}{2} \\
\frac{n_{12}}{2} & n_{22}
\end{smallmatrix} \right)$  and  $R = (r_1,r_2)$. 
From the transformation law of $\Phi$ one can conclude:
\begin{lemma}\label{well_defined_lemma}
The Fourier coefficients  $C(N,R)$  only depend upon $D_1,D_2,  D$, and the values of $r_1$ and $r_{2}$ modulo $2m$. 
In particular if $m$ is either $1$ or a prime, then they only depend on $D_1$, $D_2$, and $D$.
  \end{lemma}   
   \noindent
\textit{Two remarks.} 

\noindent 1)   
The invariants $D_1,D_2$, and $D$   are natural since in the case that  
 $m$ is $1$ or a prime, Lemma  \ref{well_defined_lemma}  is 
equivalent  to the fact that the coefficient $C(N,R)$ only depends on $ (4 N - R^t \cdot R)$.

\noindent 2)   For fixed $D_1,D_2,r_1$ and $r_2$ there  exist only 
finitely many Fourier coefficients $C(N,R)$.

 Using Lemma \ref{well_defined_lemma},   we can rewrite the Fourier expansion of $\Phi$ using the following  theta decomposition
\begin{equation} \label{thetadeco}
\Phi(\Tau,Z)
= \sum_{r_1,r_2 \pmod{2m}} \Theta_{r_1,r_2}(\Tau,Z) \, g_{r_1,r_2}(\Tau),
\end{equation}
where 
\begin{eqnarray*}
\Theta_{r_1,r_2}(\Tau,Z)&:=&
\sum_{\substack{\lambda_1,\lambda_2 \in \Z \\  \lambda_i \equiv r_i \pmod{2m} }}
e^{2 \pi i \left(\frac{\lambda_1^2}{4m}\tau +\frac{\lambda_2^2}{4m} \zeta + \frac{\lambda_1\lambda_2}{2m} u + \lambda_1 z_1+ \lambda_2 z_2 \right)   },\\
g_{r_1,r_2}(\Tau)&:=&
\sum_{D_1,D_2,D} C_{r_1,r_2}(D_1,D_2,D)\,
e^{ 2 \pi i \left( -\frac{D_1}{4m}\tau - \frac{D_2}{4m}\zeta -\frac{D}{2m}u\right) }.
\end{eqnarray*}  
Throughout we write  
$\Tau= \sm{\tau}{u}{u}{\zeta}$ and $Z=(z_1,z_2)$.
We note that the involved theta series are linear independent (see Section 3 of \cite{Zi89}).
\begin{proof}[Proof of Theorem   \ref{DualityTheorem}] 
We start with the $\Gamma$-left coset decomposition of $\Gamma \sm{1}{0}{0}{p^2}\Gamma$ 
stated in (\ref{einsp2}).  
We first consider $N_1$ and define 
\begin{equation*}
\Phi_1:= \sum_{\lambda,\mu \pmod p} \Phi|_{k,m} \left(\frac{N_1}{p},(\lambda,\mu) \right)^{\uparrow}.
\end{equation*}  
One  computes that 
\begin{eqnarray*}
\Phi_1(\Tau,Z)=\sum_{\lambda,\mu \pmod p}
 \sum_{N,R}  C(N,R)\,
e^{ 2 \pi i \left( \left(n_{11} p^2 + r_{1} \lambda p +  m \lambda^2\right)\tau + 
\left(n_{12} p + r_{2}\lambda   \right) u +   (  pr_{1} + 2 m \lambda) z_1 + r_{2} z_2 + n_{22} \zeta  \right)   }.
\end{eqnarray*} 
We make the change of variables 
 \begin{equation*}
 n_{11}' := p^2 n_{11} + p \lambda r_{1}   + \lambda^2m ,\quad
 n_{12}' :=  p n_{12} +  \lambda r_2,\quad
 n_{22}':= n_{22},\quad
 r_{1}':= p r_{1} + 2 \lambda m, \quad
 r_{2}':= r_{2}.
\end{equation*}
Since  $(p,2m)=1$  the condition $n_{11},n_{12}, n_{22}, r_1,r_2 \in \Z$  is equivalent to  
$r_1' \equiv 2 \lambda m  \pmod p$,  $D_1' \equiv 0 \pmod{p^2}$, and $D' \equiv 0 \pmod p$.
Then  we obtain the   new invariants 
$D_1'=  p^2 D_1$, $D_2'=D_2$,  $D' = pD$, $r_1'=pr_1$,  and $r_2'=r_2$. 
This yields    independent of $\lambda$ and $\mu$
\begin{eqnarray*}
C_{r_1,r_2} (D_1,D_2,D) 
= C_{\bar p r_1',r_2'} \left( \frac{D_1'}{p^2},D_2',\frac{D'}{p}\right).
\end{eqnarray*} 
Therefore we obtain  
\begin{multline*}
\Phi_1(\Tau,Z)
= p^{k+1}
\sum_{\lambda \pmod p}\quad
\sum_{ \substack{ r_1,r_2,n_{11},n_{12},n_{22} \in \Z\\ r_1\equiv 2 \lambda m \pmod p}}
C_{\bar p r_1,r_2} \left( \frac{D_1}{p^2},D_2,\frac{D}{p}\right)
e^{2 \pi i  \left( \text{tr} \left(N \Tau\right)+R Z^t \right)}\\
=
p^{k+1} 
\sum_{r_1, r_2 \pmod{2m}} \Theta_{r_1,r_2} (\Tau,Z) 
\sum_{D_1,D,D_2}  C_{\bar p r_1,r_2} \left( \frac{D_1}{p^2},D_2,\frac{D}{p}\right)
e^{ 2 \pi i \left( -\frac{D_1}{4m}\tau - \frac{D_2}{4m}\zeta -\frac{D}{2m}u\right) }.
\end{multline*} 
\noindent
We next consider $N_2(a)$, and define
\begin{equation*}
\Phi_2:= \sum_{\substack{\lambda,\mu \pmod p\\ a \pmod{p}^* }} \Phi|_{k,m} \left(\frac{N_2(a)}{p},(\lambda,\mu) \right)^{\uparrow}.
\end{equation*} 
One can prove that  
\begin{eqnarray*}
\Phi_2(\Tau,Z)=
\sum_{\substack{\lambda,\mu \pmod p\\ a \pmod{p}^*    }} \quad
 \sum_{N,R}  C(N,R)\, e^{   \frac{2 \pi i a n_{11}}{p}}\,
e^{ 2 \pi i \left( \left( n_{11} + \lambda r_{1}  + \lambda^2 m\right)\tau + 
\left(n_{12}  + \lambda  r_2 \right) u +   (  r_{1} + 2  \lambda m) z_1 + r_{2} z_2 + n_{22} \zeta  \right)   }.
\end{eqnarray*}
We  consider $\Phi^*_2$ which arises from $\Phi_2$ by completing  the sum over $a$ into a sum over $a$ modulo $p$. The new sum   over $a$ vanishes unless $p|n_{11}$ in which case it equals $p$.   
   We make the change of variables 
 \begin{eqnarray*}
 n_{11}' &:=& n_{11} + \lambda r_{1}  +\lambda^2 m ,\quad
 n_{12}' :=  n_{12} +  \lambda r_2,\quad
 n_{22}':= n_{22}, \quad
 r_{1}':= r_{1} + 2 \lambda m, \quad
 r_{2}':= r_{2}.
\end{eqnarray*}
The new invariants are  $D_1'= D_1, D_2'=D_2, 
D'=   D,
r_1' =r_1,$ and $r_2' =r_2$.
The condition  $r_1,r_2,n_{11},n_{12}, n_{22} \in \Z$ is equivalent to  
$r_1',r_2',n_{11}',n_{12}' ,  n_{22}'  \in \Z$.    
Moreover the  congruence $n_{11} \equiv 0 \pmod p$ is equivalent to the congruence
\begin{equation}  \label{congruencedisc}
4m^2 \left(\lambda - \overline{2m}r_1' \right)^2 \equiv D_1' \pmod p.
\end{equation} 
The number of solutions $\lambda$ of the congruence (\ref{congruencedisc})  equals
 $1 + \chi_{D_1}(p)$. 
  Hence  
\begin{eqnarray*}
C_{r_1,r_2} (D_1,D_2,D) 
= C_{r_1',r_2'} \left(D_1',D_2',D'\right)
\end{eqnarray*} 
and  
\begin{equation*}
\Phi_2^*(\Tau,Z)
= p^{2}  
\sum_{r_1, r_2 \pmod{2m}} \Theta_{r_1,r_2} (\Tau,Z) 
\sum_{D_1,D,D_2}  
\chi_{D_1}(p)
C_{r_1,r_2} \left( D_1,D_2,D\right) 
e^{ 2 \pi i \left( -\frac{D_1}{4m}\tau - \frac{D_2}{4m}\zeta -\frac{D}{2m}u\right) }.
\end{equation*}
\noindent
We next consider $N_3(b)$, and set
\begin{equation*}
\Phi_3:= \sum_{\substack{\lambda,\mu \pmod p\\ b \pmod{p^2} }} \Phi|_{k,m} \left(\frac{N_3(b)}{p},(\lambda,\mu) \right)^{\uparrow}.
\end{equation*}
One can show that  $\Phi_3$ equals
\begin{multline*}
\sum_{\substack{\lambda,\mu \pmod p\\ b \pmod{p^2}  }}
 \sum_{N,R}  C(N,R)\, e^{2 \pi i \left(   \frac{ b n_{11}}{p^2} +\frac{\mu r_1}{p} \right)}
e^{ 2 \pi i \left( \left( \frac{n_{11}}{p^2} + \frac{\lambda r_{1}}{p}  + \lambda^2 m\right)\tau + 
\left(\frac{n_{12}}{p}  + \lambda  r_2 \right) u +   \left(  \frac{r_{1}}{p} + 2  \lambda m\right) z_1 + r_{2} z_2 + n_{22} \zeta  \right)   }.
\end{multline*}
The   sum over  $b$ vanishes 
unless $p^2 |n_{11}$ in which case it equals $p^2$. Moreover the sum over $\mu$ vanishes unless $p|r_1$ in which case it equals $p$.  
   We make the change of variables 
 \begin{eqnarray*}
 n_{11}' := \frac{n_{11}}{p^2} +\frac{\lambda r_{1}}{p}  +\lambda^2 m,\quad
 n_{12}' :=  \frac{n_{12}}{p} + \lambda r_2,\quad
 n_{22'}:= n_{22},\quad
 r_{1}':= \frac{r_{1}}{p} + 2  \lambda m, \quad
 r_{2}':= r_{2}.
\end{eqnarray*} 
The restrictions $p^2|n_{11}, p|r_1$, and $r_2,n_{12},n_{22} \in \Z$ are 
equivalent to $n_{11}',n_{12}',n_{22}',r_1', r_2' \in \Z$. 
In particular $r_1'$ runs through $\Z$ for each $\lambda$. 
We obtain the invariants
$D_1'=  \frac{D_1}{p^2}$,  $D_2'=D_2$, 
$D'=   \frac{D}{p}$,  $r_1' =\overline{p}r_1$,  and 
$r_2' =r_2$. 
Thus   
\begin{eqnarray*}
C_{r_1,r_2} (D_1,D_2,D) 
= C_{pr_1',r_2'} \left(p^2D_1',D_2',pD'\right).
\end{eqnarray*}   
This yields 
\begin{equation*}
\Phi_3(\Tau,Z)
= p^{4-k}  
\sum_{r_1, r_2 \!\!\! \pmod{2m}} \!\!\!\!  \Theta_{r_1,r_2} (\Tau,Z) 
\sum_{D_1,D,D_2}  
C_{pr_1,r_2} \left( p^2D_1,D_2,pD\right)  
e^{ 2 \pi i \left( -\frac{D_1}{4m}\tau - \frac{D_2}{4m}\zeta -\frac{D}{2m}u\right) }.
\end{equation*} 
In a similar manner we treat   $\Phi\vert_{k,m} \left( T^J(p)\right)^{\downarrow}$. 
Now the claim of the theorem follows by comparing Fourier coefficients 
and by using  the linear independence of the theta series $\Theta_{r_1,r_2}$.
\end{proof}
\section{Proof of Theorem \ref{HeckinvTheorem}} \label{InvSection} 
Throughout we let  $\Phi \in \mathbb{E}_{k,m}^{(q)}$.
We show that for all  $S_i$ as defined in (\ref{inv}) we have  
 $\Phi|T_2^J(S_i) \in \mathbb{E}_{k,m}^{(q)}$.
 We actually  show   that 
each  package of representatives  $M_i$  which correspond to one 
of the $S_i$ already preserves the  Hecke duality. 
As a by-product we  explicitly  determine this action on the Fourier coefficients of $\Phi$.   

We first consider  the action of the Heisenberg group  on $\Phi|_{k,m}g$ with $g \in \Gamma_2\backslash \Gamma_2 S \Gamma_2$.  
For this we define $M:= \sum_{\lambda,\mu \pmod p}M(\lambda,\mu)$, where   
$\lambda:= (\lambda_1,\lambda_2), \mu:= (\mu_1,\mu_2)$, and where  
\begin{eqnarray*}
M(\lambda,\mu):= \left(
\begin{matrix}
1&0&0&0&0&\mu_1\\
0&1&0&0&0&\mu_2\\
\lambda_1&\lambda_2&1&\mu_1&\mu_2&0\\
0&0&0&1&0&-\lambda_1\\
0&0&0&0&1&-\lambda_2\\
0&0&0&0&0&1
\end{matrix}
  \right).
\end{eqnarray*}
Then   $M(\lambda,\mu) \circ \mat{\Tau}{Z^t}{Z}{w_0}$ equals 
 \begin{multline*}
 \matt{\tau}{u}{\lambda_1 \tau+\lambda_2 u+  z_1}{u}{\zeta}{\lambda_1 u +\lambda_2 \zeta+ z_2}{  \lambda_1 \tau+\lambda_2 u+  z_1   }{\lambda_1 u + \lambda_2 \zeta+z_2}
 {\lambda_1^2 \tau + 2 \lambda_1 \lambda_2 u + 2 \lambda_1 z_1 + 2\lambda_2 z_2 + \lambda_2^2 \zeta+ w_0}\\
  + \matt{0}{0}{\mu_1}{0}{0}{\mu_2}{\mu_1}{\mu_2}{\lambda_1 \mu_1 + \lambda_2 \mu_2}.
 \end{multline*}  
 It  will turn out that $\Phi|_{k,m}g$ has a Fourier expansion with  $n_{11}, n_{12},n_{22}\in \Z$ and $ r_1,  r_2 \in \frac{1}{p} \Z$. 
 The sum over $\mu$ vanishes unless $p|r_1,r_2$ in which case it equals $p^2$.    
  We  make the change of variables 
 \begin{multline*}
 n_{11}':= n_{11} + \lambda_1 r_1 +  \lambda_1^2 m,\qquad
 n_{12}':=n_{12} + \lambda_2 r_1 + \lambda_1 r_2 + 2 \lambda_1 \lambda_2m,\\
 n_{22}':=n_{22} + \lambda_2r_2 + \lambda_2^2 m,\qquad
 r_1':= r_1 + 2 \lambda_1 m,\qquad
 r_2':= r_2 + 2 \lambda_2m,
 \end{multline*}
 which doesn't change the associated invariants.  
 
  We next consider the action of the matrices $M_i$ ($1 \leq i \leq 12$). 
 We start with  \\
 $M _1:= \sum_{x,y,z \pmod{p^2}} M_1(x,y,z)$.  Then
 $$
   M_1(x,y,z) \circ \mat{\Tau}{Z^t}{Z}{w_0}
   = \matt{\frac{\tau}{p^2}}{\frac{u}{p^2}}{\frac{z_1}{p}}{\frac{u}{p^2}}{\frac{\zeta}{p^2}}{\frac{z_2}{p}}{\frac{z_1}{p}}{\frac{z_2}{p}}{w_0}
   + \matt{\frac{y}{p^2}}{\frac{x}{p^2}}{0}{\frac{x}{p^2}}{\frac{z}{p^2}}{0}{0}{0}{0}.
  $$
  The  sum over $x,y$, and $z$ vanishes unless $p^2|n_{11}, n_{12}$, and $n_{22}$ in which case it equals  $p^6$. We  make the change of variables 
   $n_{11}':=\frac{n_{11}}{p^2}, \, n_{12}':= \frac{n_{12}}{p^2},\,
  n_{22}':=\frac{n_{22}}{p^2},\, 
  r_1':=\frac{r_1}{p},$ and 
  $ r_2':=\frac{r_2}{p}.$  
  Observe that $n_{11}', n_{12}'$, and $n_{22}'\in \Z$ and $r_1',r_2'\in \frac{1}{p} \Z$. 
  We obtain the invariants 
   $D_1'=\frac{D_1}{p^2},\, 
  D_2'=\frac{D_2}{p^2},$
  $ D'=\frac{D}{p^2}$ 
    which yields
  \begin{eqnarray*}
C(D_1,D_2,D) 
= C\left(p^2D_1',p^2D_2',p^2D'\right).
\end{eqnarray*}
From the above  considerations we see that   applying the  Heisenberg part reduces the summation to $r_1$ and $r_2 \in \Z$, multiplies the sum by $p^2$, and leaves   the invariants unchanged.  
Thus  $\Phi|_{k,m}M_1M(\Tau,Z)$ equals
\begin{equation*}  
p^{10-2k}  
\sum_{\substack{D_1,D,D_2\\ r_1,r_2}}  
 C  \left(p^2D_1,p^2D_2,p^2D\right)
e^{ 2 \pi i \left( -\frac{D_1}{4m}\tau - \frac{D_2}{4m}\zeta -\frac{D}{2m}u\right) } 
e^{2 \pi i \left(\frac{r_1^2}{4m}\tau +\frac{r_2^2}{4m} \zeta + \frac{r_1 r_2}{2m} u + r_1 z_1+ r_2 z_2 \right)   },
\end{equation*} 
where here in the following  we have as before 
$D_1=r_1^2-4n_{11}m$, $D_2=r_2^2-4n_{22}m$, and \\
$D=r_1r_2-2n_{12}^2$.
Let  
\begin{equation*}
A(D_1,D_2,D):= 
 C  \left(p^2D_1,p^2D_2,p^2D\right).
\end{equation*}
We show that this function satisfies (\ref{Hecke_Duality_Relations}) using  that $\Phi$ satisfies 
(\ref{Hecke_Duality_Relations}) and that  $\chi_{D_1}(q) = \chi_{p^2 D_1}(q)$. This yields that  
$\left( \chi_{{ D_1}}(q) -\chi_{D_2}(q) \right) \,\, 
A \left( D_1,  D_2,D  \right)$ equals 
\begin{multline*} 
 \left( \chi_{{ D_1p^2}}(q) -\chi_{D_2p^2}(q) \right) \,\, 
C \left( p^2D_1, p^2 D_2,p^2D  \right) \\
=
q^{2-k}\, \left( C\left( p^2D_1,p^2q^2D_2 ,  p^2 qD  \right) - 	
C \left( q^2 p^2D_1 ,  p^2D_2,p^2qD   \right) \right)\\
+
q^{k-1} \left( C \left(p^2D_1,  \frac{p^2D_2}{q^2},  \frac{p^2D}{q}  \right)
-
C\left( \frac{p^2D_1}{q^2}, p^2D_2, \frac{p^2D}{q} \right) \right)\\
= q^{2-k}\, \left( A\left( D_1,q^2D_2 ,  qD  \right) - 	A \left( q^2 D_1 , D_2,qD   \right) \right)\\
+
q^{k-1} \left( A \left(D_1,  \frac{D_2}{q^2},  \frac{D}{q}  \right)
-
A\left( \frac{D_1}{q^2}, D_2, \frac{D}{q} \right) \right)
\end{multline*}
as claimed.
The matrices $M_2,\, M_3,\, M_4$, and $M_5$ are treated in a similar way.
We next  consider     $M_6:= \sum_{\substack{a \pmod p\\ b \pmod{p}^*}} M_6(a,b)$. We have
  $$
   M_6(a,b) \circ \mat{\Tau}{Z^t}{Z}{w_0}
   = 
   \matt{p^2 \tau}{- a p \tau + pu}{pz_1}{  - a p \tau + pu}{a^2 \tau-2 au + \zeta}{- a z_1+z_2}{pz_1}{- a z_1+z_2}{w_0}
   + \matt{0}{0}{0}{0}{\frac{b}{p}}{0}{0}{0}{0}.
  $$
  Since we will see later that $M_{11}$ preserves  (\ref{Hecke_Duality_Relations}) we may complete the sum over $b$ into a sum over all $b$ modulo $p$ which we denote by $M_6^*$.   The sum over $b$  vanishes unless $p|n_{22}$ in which case it equals  $p$.   
  We  make the change of variables  
  \begin{eqnarray*}
  n_{11}':=    p^2 n_{11} - a p n_{12} + a^2 n_{22},\quad
  n_{12}':=p n_{12} -2a n_{22},\quad
  n_{22}':=n_{22},\quad
  r_1':=p r_1- a r_2,\quad
  r_2':=r_2.
  \end{eqnarray*}  
  We compute the invariants  
  \begin{eqnarray*}
  D_1=\frac{1}{p^2} \left(D_1'+2 a D' + a^2 D_2' \right)   ,\quad
  D_2=D_2',\quad
  D=\frac{1}{p} \left(D' + a D_2' \right),\quad
  r_1 = \frac{1}{p} \left(r_1'+ a r_2' \right)  ,\quad
  r_2 =r_2'.
  \end{eqnarray*} 
  Applying  the Heisenberg transformation preserves those invariants. 
  We denote the new variables with tildes.  
  We have   the following equivalent conditions:
  \begin{eqnarray*}
  p|n_{22} \quad
  \Leftrightarrow \quad p|n_{22}'  \quad
 \Leftrightarrow \quad \widetilde{n}_{22} - \lambda_2 \widetilde{r}_2 + \lambda_2^2 m \equiv 0 \pmod p\quad
 \Leftrightarrow  \quad
  4m^2 \left( \lambda - \overline{2m} \widetilde{r}_2\right)^2 
  \equiv \widetilde{D}_2 \pmod p .
  \end{eqnarray*}  
   Then    $\Phi|_{k,m}M_6^*M(\Tau,Z)$ equals
\begin{multline*}  
p^{4+k}
 \sum_{\substack{D_1,D,D_2\\r_1,r_2}}  
\sum_{  \substack{ \lambda_2, a \pmod p\\   4m^2 \left( \lambda_2 - \overline{2m} r_2\right)^2 
  \equiv D_2 \pmod p  } }  
  C \left(  \frac{1}{p^2} (D_1+ 2 a D + a^2 D_2), D_2, \frac{1}{p} (D+ a D_2)\right)\\
  e^{ 2 \pi i \left( -\frac{D_1}{4m}\tau - \frac{D_2}{4m}\zeta -\frac{D}{2m}u\right) } 
  e^{2 \pi i \left(\frac{r_1^2}{4m}\tau +\frac{r_2^2}{4m} \zeta + \frac{r_1 r_2}{2m} u + r_1 z_1+ r_2 z_2 \right)   }
  e^{2 \pi i \left(\frac{r_1^2}{4m}\tau +\frac{r_2^2}{4m} \zeta + \frac{r_1 r_2}{2m} u + r_1 z_1+ r_2 z_2 \right)   }.
\end{multline*}  
We let
\begin{equation} \label{defineA}
A(D_1,D_2,D):= 
\sum_{  \substack{ \lambda, a \pmod p\\   4m^2 \left( \lambda - \overline{2m} r_2\right)^2 
  \equiv D_2 \pmod p  } }  
  C \left(  \frac{1}{p^2} (D_1+ 2 a D + a^2 D_2), D_2, \frac{1}{p} (D+ a D_2)\right)
\end{equation}
and show that this function satisfies (\ref{Hecke_Duality_Relations}). 
First observe that in the sum over $a$ in  (\ref{defineA}) we may choose as   a set of representatives    elements that are   divisible by $q$.    
Using this  we have 
$$
\chi_{D_1}(q) =  \chi_{   \frac{1}{p^2} (D_1+ 2 a D + a^2 D_2)      }(q).
$$ 
This   yields that   $\left( \chi_{D_1}(q) - \chi_{D_2}(q) \right)A(D_1,D_2,D)$ equals
\begin{multline*}
\left( \chi_{D_1}(q) - \chi_{D_2}(q) \right)
 \sum_{  \substack{ \lambda, a \pmod p\\   4m^2 \left( \lambda - \overline{2m} r_2\right)^2 
  \equiv D_2 \pmod p  } }  
  C \left(  \frac{1}{p^2} (D_1+ 2 a D + a^2 D_2), D_2, \frac{1}{p} (D+ a D_2)\right)
  \\
=
 \sum_{  \substack{ \lambda, a \pmod p\\   \substack{4m^2 \left( \lambda - \overline{2m} r_2\right)^2 
  \equiv D_2 \pmod p \\ q|a  }} }  
   \left( \chi_{   \frac{1}{p^2} (D_1+ 2 a D + a^2 D_2)      }(q) - \chi_{D_2}(q) \right)\\
  C \left(  \frac{1}{p^2} (D_1+ 2 a D + a^2 D_2), D_2, \frac{1}{p} (D+ a D_2)\right).
\end{multline*}
Using (\ref{Hecke_Duality_Relations})   gives
that  this   equals   
\begin{multline*}
q^{2-k} 
 \sum_{  \substack{ \lambda, a \pmod p\\   \substack{4m^2 \left( \lambda - \overline{2m} r_2\right)^2 
  \equiv D_2 \pmod p \\ q|a  }} } 
  C \left(  \frac{1}{p^2} (D_1+ 2 a D + a^2 D_2), q^2D_2, \frac{q}{p} (D+ a D_2)\right) \\
  -  
  q^{2-k} 
 \sum_{  \substack{ \lambda, a \pmod p\\   \substack{4m^2 \left( \lambda - \overline{2m} r_2\right)^2 
  \equiv D_2 \pmod p \\ q|a  }} }   C \left(  \frac{q^2}{p^2} (D_1+ 2 a D + a^2 D_2), D_2, \frac{q}{p} (D+ a D_2)\right)
  \\
  +
  q^{k-1} 
 \sum_{  \substack{ \lambda, a \pmod p\\   \substack{4m^2 \left( \lambda - \overline{2m} r_2\right)^2 
  \equiv D_2 \pmod p \\ q|a  }} } 
  C \left(  \frac{1}{p^2} (D_1+ 2 a D + a^2 D_2), \frac{D_2}{q^2}, \frac{1}{pq} (D+ a D_2)\right)  \\ 
  -  
    q^{k-1} 
 \sum_{  \substack{ \lambda, a \pmod p\\   \substack{4m^2 \left( \lambda - \overline{2m} r_2\right)^2 
  \equiv D_2 \pmod p \\ q|a  }} }   C \left(  \frac{1}{p^2q^2} (D_1+ 2 a D + a^2 D_2), D_2, \frac{1}{pq} (D+ a D_2)\right)
  .
\end{multline*} 
We  rewrite the occuring summands on the right hand side. 
Since $(p,q)=1$, we may change in the first sum $a$ into $aq$ and in the second $a$ into $aq$ and $\lambda$ into $\lambda \bar{q}$. The other summands are  treated similarly.  
  This gives that  $\Big( \chi_{{ D_1}}(q) -\chi_{D_2}(q) \Big) \,\, A \left( D_1,  D_2,D  \right)$ equals
  \begin{equation*}   
q^{2-k}\, \left( A \left( D_1,D_2 q^2,  D q \right) - A \left( D_1 q^2,  D_2,D q  \right) \right)
+
q^{k-1} \left( A\left(D_1,  \frac{D_2}{q^2},  \frac{D}{q}  \right)
-
A \left( \frac{D_1}{q^2}, D_2, \frac{D}{q} \right) \right)
\end{equation*}
as claimed.

 We next  deal with  the action of      
$M_7:= \sum_{ \substack{   b,c \pmod p\\ a \pmod p^*     }} M_7(a,b,c)$. We have 
  $$
   M_7(a,b,c) \circ \mat{\Tau}{Z^t}{Z}{w_0}
   =  
   \matt{\tau}{u - \frac{a}{p} \tau}{z_1}{u-\frac{a}{p}\tau }{\frac{a^2}{p^2} \tau - 2 \frac{a}{p} u + \zeta}{z_2 -\frac{a}{p} z_1 }{z_1}{z_2 -\frac{a}{p} z_1}{w_0}
      + \matt{0}{\frac{c}{p}}{0}{\frac{c}{p}}{\frac{bp-ca}{p^2} }{0}{0}{0}{0}.
  $$
  The sum over $b$ vanishes unless
  $n_{22} \equiv 0 \pmod p$. 
  Moreover the sum over $c$ vanishes unless \\ 
  $n_{12} -\frac{a n_{22}}{p}  \equiv 0 \pmod p$. 
     We make the change of variables
  \begin{multline*}
  n_{11}':=    n_{11} - \frac{a}{ p} n_{12} + \frac{a^2}{p^2} n_{22},\quad
  n_{12}':= n_{12} -2\frac{a}{p} n_{22}, \quad
  n_{22}':=n_{22},\quad
  r_1':=r_1- \frac{a}{p} r_2,\quad
  r_2':=r_2.
  \end{multline*}   
  Then  
  \begin{equation*}
  D_1=  D_1'+\frac{2a}{p}D' + \frac{a^2}{p^2}   D_2' ,\qquad
  D_2=D_2',\qquad
  D= D'+\frac{a}{p}D_2' .
    \end{equation*} 
    We  again denote the variables after  the Heisenberg transformation with tildes.
    As before we see that the condition $n_{22} \equiv 0 \pmod p$ is equivalent to  
    \begin{equation}\label{congruence1}
    4m^2 \left( \lambda_2 - \overline{2m} \widetilde{r}_2\right)^2 
  \equiv \widetilde{D}_2 \pmod p .
    \end{equation}
    Similarly we have  the equivalence 
      \begin{equation*}
  n_{12} -\frac{an_{22}}{p}  \equiv 0 \pmod p \quad \Leftrightarrow \quad  
  n_{12}' +\frac{a}{p}n_{22}' \equiv 0 \pmod p,
  \end{equation*} 
  which is  equivalent to 
  \begin{equation} \label{congruence2}
 \widetilde{n}_{12} - \lambda_2 \widetilde{r}_1 - \lambda_1 \widetilde{r}_2 + 2 \lambda_1 \lambda_2 m + \frac{a}{p}  \left(\widetilde{n}_{22} - \lambda_2 \widetilde{r}_2 + \lambda_2^2m\right) \equiv 0 \pmod p.
  \end{equation} 
  This yields    that $\Phi|_{k,m}M_7M(\Tau,Z) $ equals
  \begin{multline*}
p^{2} 
\sum_{\substack{D_1,D,D_2\\r_1,r_2}}  \quad
\sump_{  \substack{\lambda_1, \lambda_2 \pmod p\\    a \pmod{p}^*   } }  
  C \left(  D_1+ \frac{2 a D}{p} + \frac{a^2 D_2}{p^2}, D_2, D+ \frac{a D_2}{p}\right)
  e^{ 2 \pi i \left( -\frac{D_1}{4m}\tau - \frac{D_2}{4m}\zeta -\frac{D}{2m}u\right) } 
   e^{2 \pi i \left(\frac{r_1^2}{4m}\tau +\frac{r_2^2}{4m} \zeta + \frac{r_1 r_2}{2m} u + r_1 z_1+ r_2 z_2 \right)   },
\end{multline*}  
where in $\sump$ the sum runs over those $\lambda_1,\lambda_2$ and $a$ that satisfy (\ref{congruence1}) and  (\ref{congruence2}).
Now we can argue as before.
The case of $M_8$ is proven similarly. 

  We next  consider the action of    
  $M_9:=\sum_{\substack{ \substack{a,z \pmod p \\ x  \pmod p^*\\  \kappa \pmod{p^2}           }}}M_9(a,x,z,\kappa)$. We compute 
    $$
   M_9(a,x,x,\kappa)
   \circ \mat{\Tau}{Z^t}{Z}{w_0}
   =  
\matt{\tau}{- \frac{a \tau}{p} +\frac{u}{p}}{z_1}{- \frac{a \tau}{p} +\frac{u}{p}}{\frac{a^2 \tau}{p^2} - 
\frac{2au}{p^2} +\frac{\zeta}{p^2}  }{-\frac{az_1}{p}+\frac{z_2}{p} }{z_1}{-\frac{az_1}{p} +\frac{z_2}{p}}{w_0} 
      + \matt{\frac{x}{p}}{\frac{z}{p}}{0}{\frac{z}{p}}{\frac{\kappa}{p^2}-\frac{az}{p^2}}{0}{0}{0}{0}.
  $$
   Since it  turns out  that $M_{12}$ preserves  (\ref{Hecke_Duality_Relations}) we may complete the sum over $x$ into a sum over all $x$ modulo $p$.  
  The sum over $x,z$, and $\kappa$  vanishes unless 
    $p|n_{11},n_{12}$ and $p^2|n_{22}$ in which case it equals $p^4$.   
       We make the change of variables
  \begin{eqnarray*}
  n_{11}':=    n_{11} - \frac{a  n_{12} }{p}+ \frac{a^2}{p^2} n_{22},\quad
  n_{12}':= \frac{n_{12}}{p} -2\frac{a}{p^2} n_{22},\quad
  n_{22}':=\frac{n_{22}}{p^2}, \quad
  r_1':=r_1- \frac{a}{p} r_2,\quad
  r_2':=\frac{r_2}{p}
  \end{eqnarray*} 
  and   obtain
  \begin{equation*}
  D_1=  D_1'+ 2a D' + a^2   D_2' ,\qquad
  D_2=p^2D_2',\qquad
  D= p(D'+aD_2')  .
  \ \end{equation*} 
   This yields     
     \begin{multline*} 
     \Phi|_{k,m}M_9M(\Tau,Z)
 p^{8-k} 
\sum_{\substack{D_1,D,D_2\\r_1,r_2}}\,\,  
\sum_{   a\!\!\!\!\!\! \pmod p^* }  
  C \left(  D_1+ 2a D+ a^2   D_2, p^2D_2, p(D+ a D_2)\right) 
  \\
  e^{ 2 \pi i \left( -\frac{D_1}{4m}\tau - \frac{D_2}{4m}\zeta -\frac{D}{2m}u\right) } 
   e^{2 \pi i \left(\frac{r_1^2}{4m}\tau +\frac{r_2^2}{4m} \zeta + \frac{r_1 r_2}{2m} u + r_1 z_1+ r_2 z_2 \right)   }.
\end{multline*}  
Now we  can argue as before.  
The matrices $M_{10}$, $M_{11}$, and  $M_{12}$ can be considered similarly.  
\section{Proof of Theorem  \ref{EisenTheorem}}\label{EisenSection}
As  explicit examples of  Jacobi forms of degree $2$ which  are elements  of $\mathbb{E}_{k,m}$,
 we define Jacobi Eisenstein series of Siegel type
\begin{eqnarray*}
E_{k,m}^{J,n}(\Tau,Z):= \sum_{\gamma \in   \Gamma_{\infty}^J   \backslash \Gamma_n^J   }
J^{-1}_{k,m}(\gamma, (\Tau,Z)).
\end{eqnarray*}
Here   $\Gamma_{\infty}^J$ is the stabilizer group of the function $J_{k,m}^{-1}$.
  The series $E_{k,m}^{J,n}$ is 
absolutely convergent for $k>n+2$ and defines a non-trivial Jacobi form.  
Since related functions  occur in work of Arakawa \cite{Ar94}, 
we choose for the readers convenience to  use his  parametrization of the Jacobi group in terms of matrices    to simplify consulting 
this paper for related calculations.
 We note that the Eisenstein series viewed as functions are the same.
Only formally the sets of representatives of cosets defining the Eisenstein series have a different parametrization which can be related to each other by conjugation. 

We have  that  
$\Gamma_{\infty}^J= \left\{ \gamma \in \Gamma_n^J \vert \,\lambda = 0, \, g \in \Gamma_{\infty}
\right\}, $   where   $\Gamma_{\infty}$ is  the subgroup of $\Gamma_n$ with $C=0$. 
It follows from the definition of the Eisenstein series that   
\begin{equation*}
E_{k,m}^{J,n }(\Tau,Z) 
= \sum_{ \lambda \in \Z^n} \sum_{ g \in \Gamma_{\infty }  \backslash \Gamma_n}
J_{k,m}^{-1} \Big(
 \left( \left( \lambda,0,0  \right),g\right),
 (\Tau,Z) \Big).
\end{equation*}   
In the following we  restrict to   $n=2$ and  
 analyze the decomposition of $\Gamma_{\infty} \backslash \Gamma_2$ 
with respect to $\Gamma  \times \Gamma$.   
Work of Garrett \cite{Ga84}  implies that  
\begin{equation*}
\Gamma_2 = \Gamma_{\infty} (\Gamma \times \Gamma) \, \cup \, 
\bigcup_{d=1}^{\infty} \Gamma_{\infty} \,\,h_d'\, \,\Gamma \times \Gamma, 
\end{equation*}
where 
$h_d' := \left(  \begin{smallmatrix}1&0&0&0\\ 0&1&0&0\\ 
0&d&1&0\\ d&0&0&1 \end{smallmatrix} \right).$  
A straightforward calculation gives that $\Gamma_{\infty} \backslash \Gamma_2$ can be written as the union of 
$\left(\Gamma_{\infty  } \backslash \Gamma\right) \times \left(\Gamma_{\infty} \backslash  \Gamma\right)$ and
$\bigcup_{d=1}^{\infty} h_d' \left(\left(
\Gamma(d) \backslash \Gamma \right) \times \Gamma\right)$, where 
$\Gamma(d):= \left(\begin{smallmatrix} d^{-1}& 0\\0 &d   
\end{smallmatrix}  \right)  \Gamma  \left(\begin{smallmatrix} 
d& 0\\0 &d^{-1}   \end{smallmatrix}  \right) \cap \Gamma$.  
Since $I_2 \times \left(\begin{smallmatrix} 0& -1\\1 &0   
\end{smallmatrix}  \right)$ is an element of $I_2 \times \Gamma$,
we can replace $h_d'$ by 
$$h_d:=
h_d'  \, \left( I_2 \times \left(\begin{smallmatrix} 0& -1\\1 &0   \end{smallmatrix}  \right)\right) = 
 \left(  \begin{smallmatrix}1&0&0&0\\ 0&0&0&-1\\ 0&0&1&-d\\ d&1&0&0 \end{smallmatrix} \right). 
$$  
We denote the subseries  corresponding to $\left(\Gamma_{\infty} 
\backslash \Gamma \right) \times  \left(\Gamma_{\infty} \backslash \Gamma \right) \text{ and }
\bigcup_{d=1}^{\infty} h_d' \big(\left(
\Gamma(d) \backslash \Gamma \right) \times \Gamma\big)$ by $E_{I}$ and $E_{II}$, respectively.  
One computes that  
\begin{equation*}
E_{I}(\Tau,Z)= \sum_{ \substack{ \lambda \in \Z^2\\ g,h \in \Gamma_{\infty} \backslash \Gamma}}
J_{ k,m}^{-1} \left( (\lambda,0,0) g \times h  ), (\Tau,Z) \right).
\end{equation*}
The   conjugation law of the Heisenberg group  implies that  
\begin{equation*}
h_d^{-1} \big((x,y),(0,0),0\big) h_d
=(x,0,0) \times (0,-y,0).
\end{equation*}
This yields 
\begin{equation*}
E_{II}(\Tau,Z)= \sum_{d=1}^{\infty } \sum_{ \gamma_1,\gamma_2} 
J_{ k,m}^{-1} \left( h_d \cdot ( \gamma_1 \times \gamma_2 ), (\Tau,Z) \right),
\end{equation*}
where 
$\gamma_1 \in \big( 
(\Z,0,0), 
\big( \Gamma(d) \backslash \Gamma \big) \times I_2  \big)$ and 
$\gamma_2\in\big( (0,\Z,0) ,   I_2  \times \Gamma \big)$.
We treat the subseries $E_I$ and $E_{II}$  separately.
\subsection{The subseries  $E_I$}
			\begin{proposition} 
			The series    
$E_{I}(\Tau,Z)$ satisfies  (\ref{Hecke_Duality_Relations})   for all primes $p$ with $(p,2m)=1$. 
If   $m$ is  square-free, then  this is true for   all primes $p$.  
			\end{proposition} 
\begin{proof}
The cocyle relation of $J_{k,m}$ yields
\begin{equation}\label{num}
E_{I}(\Tau,Z)
 =
\sum_{ \gamma, \widetilde{\gamma} \in \Gamma_{\infty}^J \backslash \Gamma^J}
J_{ k,m}^{-1} \left( 
\gamma \times \widetilde{\gamma} , (\Tau,Z) 
\right) 
=
\sum_{ \widetilde{\gamma} \in \Gamma_{\infty}^J \backslash \Gamma^J}
E_{k,m}^J \left( 
\left( \widetilde{\gamma}^\downarrow \circ (\Tau,Z)\right)^{*} 
\right) 
J_{ k,m}^{-1} \left( 
\widetilde{\gamma}^\downarrow , (\Tau,Z) 
\right),
\end{equation}
where $(\Tau,Z)^{\ast} := (\tau,z_1) \in \H^J$.    
Without loss of generality we may assume that   
$T^J(p) = 
\sum_j \Gamma^J \, \eta_j$, where $p$ is a prime 
and $\eta_j \in G^J(\Q)$. Using (\ref{num}) gives that $\left(E_I \vert T^J(p)^\uparrow \right) (\Tau, Z)$ equals
\begin{multline*}
 \sum_{ \widetilde{\gamma} \in \Gamma_{\infty}^J \backslash \Gamma^J} \sum_j
E_{k,m}^J \left( 
\left( \widetilde{\gamma}^\downarrow \eta_j^\uparrow \circ (\Tau,Z)\right)^{*} 
\right)
J_{ k,m}^{-1} \left( 
\widetilde{\gamma}^\downarrow , \eta_j^\uparrow \circ (\Tau,Z) 
\right)
\,\,
J_{ k,m}^{-1} \left( 
\eta_j^\uparrow, (\Tau,Z) 
\right)  \\
=  
\sum_j
\sum_{ \widetilde{\gamma} \in \Gamma_{\infty}^J \backslash \Gamma^J} 
E_{k,m}^J 
\Big( 
\eta_j \circ 
\left( \widetilde{\gamma}^\downarrow \circ (\Tau,Z)\right)^{*} 
\Big)   
\,\,J_{ k,m}^{-1} \left( 
\eta_j, 
\left(\widetilde{\gamma}^\downarrow \circ (\Tau,Z) \right)^{\ast}
\right)
\,\,
J_{ k,m}^{-1} \left( 
\widetilde{\gamma}^\downarrow ,(\Tau,Z) 
\right). 
\end{multline*} 
This yields  
\begin{eqnarray*}
\left(E_I \vert T^J(p)^\uparrow \right) (\Tau, Z) 
= 
\sum_{ \widetilde{\gamma} \in \Gamma_{\infty}^J \backslash \Gamma^J}
\left(E_{k,m}^J \vert T^J(p)\right)
\left( 
\left( \widetilde{\gamma}^\downarrow \circ (\Tau,Z)\right)^{*} 
\right)  \,\,
J_{ k,m}^{-1} \left( 
\widetilde{\gamma}^\downarrow ,(\Tau,Z) 
\right). 
\end{eqnarray*}
By formula $(13)$ in \cite{EZ85} it follows that for $(p,m)=1$ the  
 Jacobi Eisenstein series of degree $1$ and index $m$
is a Hecke Jacobi eigenform.  
If $m$ is square-free it follows from  formula ($51$) in the table on page 224 of \cite{He99} 
   that  this also true if  $p|m$. Thus 
     the claim of the proposition  follows.
\end{proof}  
\subsection{The subseries $E_{II}$}
In this section we prove an explicit formula  for  $E_{II}$ as an infinite sum of certain functions
$P_{k,m}^J \left( \Tau, Z \right) $ of Poincar\'{e} type. 
For this purpose define   the function
\begin{eqnarray*}
\psi_{k,m}(\Tau,Z) :=  (\tau+2u+\zeta)^k\, e \left(\frac{m(z_1+z_2)^2}{\tau+2u +\zeta}  \right),
\end{eqnarray*} 
which  is   related to the Jacobi Eisenstein series.
A direct  calculation leads  to
				\begin{eqnarray} \label{simplification}
				J_{k,m}(h_d,(\Tau,Z)  )=   \psi_{k,m} 
				(\M_d\circ (\Tau,Z)).
				\end{eqnarray}   
				Here $\M_d:=  \left( \begin{smallmatrix} 
				d&0\\0&d^{-1}   \end{smallmatrix}  \right) \times I_2$.
Next  define the  unique  involution  $\#: G^J(\R) \to G^J(\R)$  
on the generators of the Jacobi group by  $(\lambda,0,0)^{\#} := (-\lambda, 0,0)$, 
$(0,\mu,\kappa)^{\#} := (0, \mu,\kappa)$, 
$\left(\begin{smallmatrix} a&b\\ c&d \end{smallmatrix} \right)^{\#} :=
\left(\begin{smallmatrix} d&b\\ c&a \end{smallmatrix} \right)$ that   satisfies 
$(\gamma_1 \gamma_2 )^{\#} := \gamma_2^{\#} \gamma_1^{\#}$ for all
$\lambda, \mu, \kappa \in \R$ and $\gamma_1,\gamma_2 \in G^J(\R)$.  
This involution preserves double cosets in the Hecke Jacobi algebra.
A straightforward but lengthy  calculation gives.
\begin{lemma}\label{dual}
For $\gamma \in G^J(\R)$ we have 
\begin{eqnarray*}
\psi_{k,m}^{-1} | \gamma^{\uparrow} = \psi_{k,m}^{-1}| \left. \gamma^{\#} \right.^{\downarrow}.
\end{eqnarray*}   
\end{lemma} 
We next define for     $D \in Sp_1(\Q)$   the holomorphic function
\begin{equation*}
P_{k,m}^J \left(\left( \Tau, Z \right), D\right) := 
\sum_{\gamma \in Z(\Gamma^J) 
\backslash \Gamma^J \, D \, \Gamma^J} \left( \psi_{k,m}^{-1} \vert 
\gamma^\uparrow \right) \left( \Tau,Z\right),  
\end{equation*}
where $Z( \Gamma^J)$ is the center of $\Gamma^J$. 
This series is absolutely and locally uniformly convergent.  
Its restriction   to  $\H^J \times \H^J$ yields   the Poincar\'{e} series given
in \cite{Ar94}. 
Moreover  there exist a unique   $d \in\N$ with 
$D = 
\left(\begin{smallmatrix}  
d & 0 \\ 0 & d^{-1} \end{smallmatrix}\right)$.  
Lemma  \ref{dual}   
and the 
invariance of double cosets with respect to the involution $\#$ gives that 
\begin{equation*}
P_{k,m}^J \left(\left( \Tau, Z \right), D\right) = \sum_{\gamma \in Z(\Gamma^J) 
\backslash \Gamma^J \, D \, \Gamma^J} \left( \psi_{k,m}^{-1} \vert \gamma^\downarrow \right) \left( \Tau,Z\right).
\end{equation*}
Moreover for  $U \in SL_2(\Z)$, we define   $P_{k,m}^J \left(\Tau, Z\right) : = P_{k,m}^J \left(\left( \Tau, Z \right), U\right)$ . Then we have 
\begin{equation}\label{transform}
\left(P_{k,m}^J\vert
\mathbb{X}(d)^\uparrow \right)
 \left(\Tau, Z\right) = P_{k,m}^J \left(\left( \Tau, Z \right), D\right).
\end{equation}   
\begin{lemma}
Assume that $m$ is square-free. Then we have  
\begin{equation} \label{obenunten}
\left(P_{k,m}^J\vert
\mathbb{X}(l)^\uparrow\right)
 \left(\left(\Tau, Z\right), 
\left(\begin{smallmatrix}  
d & 0 \\ 0 & d^{-1} \end{smallmatrix}\right) \right) = 
\left(P_{k,m}^J\vert
\mathbb{X}(l)^\downarrow \right)
 \left(\left(\Tau, Z\right), 
\left(\begin{smallmatrix}  
d & 0 \\ 0 & d^{-1} \end{smallmatrix}\right) \right).
\end{equation}  
If $m$ is arbitary (\ref{obenunten}) is also satisfied if $(l,d)=1$ or $(ld,m)=1$.
\end{lemma} 
\begin{proof}  
The conditions that $m$ is square-free and that $(l,d)=1$ or $(ld,m)=1$ otherwise, imply that  
on $J_{k,m}$ we have 
\begin{equation*}
\mathbb{X}(l) \, \mathbb{X}(d)  = \mathbb{X}(d) \,\,\mathbb{X}(l). 
\end{equation*}
 Finally we apply formula (\ref{transform}) to get the lemma.
\end{proof} 
					\begin{proposition} 
					We have
					\begin{equation}
					E_{II} \left( \Tau,Z \right) = \sum_{d=1}^{\infty} 
					\left( P_{k,m}^J \vert \mathbb{X}(d)^\uparrow \right) (\Tau,Z) \,\, d^{-k}.
					\end{equation}
					\end{proposition}
\begin{proof}
Denote by   $E_{II}^d $ the subseries of $E_{II}$ corresponding to 
$d \in \N$. Then we have by  (\ref{simplification})
\begin{equation*}
E_{II}^d(\Tau,Z) = \sum_{\gamma_1,\gamma_2} \psi_{k,m}^{-1} 
\left(
\M_d \, (\gamma_1 \times \gamma_2) \circ (\Tau,Z) \right)
\,\,
J_{k,m}^{-1} \left( \gamma_1 \times \gamma_2, (\Tau,Z) \right).
\end{equation*}
Here
$\gamma_1 \in \Big( 
(\Z,0,0), 
\Gamma(d) \backslash \Gamma \times I_2 \Big)$ and 
$\gamma_2  \in \Big( (0,\Z,0),    I_2  \times \Gamma \Big)$. 
Using Lemma \ref{dual},   a  
 straightforward, but lengthy calculation gives
\begin{equation*}
E_{II}^d(\Tau,Z) = d^{-k} \sum_{\gamma_1,\gamma_2} \psi_{k,m}^{-1} 
\left(
\left(\gamma_2^{\#} 
\left( \begin{smallmatrix} d & 0 \\ 0 & d^{-1} \end{smallmatrix} \right) 
\gamma_1\right)^\uparrow \circ (\Tau,Z)
\right)
J_{k,m}^{-1} \left( \left(
\gamma_2^{\#} \left( \begin{smallmatrix} d & 0 \\ 0 & d^{-1} 
\end{smallmatrix} \right) \gamma_1\right)^\uparrow , (\Tau,Z) \right).
\end{equation*}
To complete the proof it   remains to analyze the set
\begin{equation}\label{set}
\Big\{\gamma_2^{\#} \left( \begin{smallmatrix} d & 0 \\ 0 & d^{-1} \end{smallmatrix} \right) \gamma_1
\Big{\vert} \,\,
\gamma_2 = \left( (0,\mu,0)g\right), \, \gamma_1 = \left( (\lambda,0,0)h \right) 
\text{ with } 
g \in \Gamma, \, h \in \Gamma(d) \backslash \Gamma \text{ and } \lambda, \mu \in \Z
\Big\}.
\end{equation}
First we note that $\Big((0,\Z,0)\, \Gamma\Big)^{\#} =\Gamma \,(0,\Z,0)$.   
Hence the set in (\ref{set}) is equal to \\
$\Gamma \,\, (0,\Z,0) \,\left( \begin{smallmatrix} d & 0 \\ 0 & d^{-1} \end{smallmatrix} 
\right) (\Z,0,0) \,\,\Gamma(d) \backslash \Gamma$, which equals  
\begin{equation*}
\Gamma \,(0,\Z,0) \,\, (\Z,0,0) \left( \begin{smallmatrix} d & 0 \\ 0 & d^{-1} \end{smallmatrix} 
\right) (d \Z \backslash \Z,0,0) \Gamma(d) \backslash \Gamma,
\end{equation*}
since 
$\left( \begin{smallmatrix} d^{-1} & 0 \\ 0 & d \end{smallmatrix} \right) 
( \lambda',0,0) 
\left( 
\begin{smallmatrix} d & 0 \\ 0 & d^{-1} \end{smallmatrix} \right) = (d \lambda',0,0)$. Here
$\lambda'$ runs modulo $d$. 
Hence the set (\ref{set}) is equal to 
$Z(\Gamma^J) \backslash \Gamma^J \left( 
\begin{smallmatrix} d & 0 \\ 0 & d^{-1} \end{smallmatrix} \right) \Gamma^J.
$
Thus
\begin{equation}
E_{II}^d (\Tau,Z) = d^{-k} \sum_{ \gamma \in Z(\Gamma^J) \backslash \Gamma^J \left( 
\begin{smallmatrix} d & 0 \\ 0 & d^{-1} \end{smallmatrix} \right) \Gamma^J}
\psi_{k,m}^{-1} \left( \gamma^\uparrow \circ (\Tau,Z)\right) \,\, J_{k,m}^{-1} \left( \gamma^\uparrow , (\Tau,Z) \right),
\end{equation}
which  leads to the proof of the proposition.
\end{proof}
\section{Eisenstein series of Klingen type} 
In this final section we show that   for even $k>4$, the space $\mathbb{E}_{k}$ is  a
proper  subspace of  $J_{k,m}^{2}$.   
Let 
$\Phi \in J_{k,m}^{\text{cusps}}$, the space of Jacobi cusp forms,  a non-trivial   
Hecke-Jacobi eigenform for all $T^J(l)$ with $(l,m)=1$.
Denote by   $
E_{k,m}^{\text{Kl}}\left( \Phi, (\Tau,Z) \right) 
$
the  Jacobi Klingen Eisenstein series of degree $2$ associated to $\Phi$.  
 Then $E_{k,m}^{\text{Kl}}\left( \Phi \right)$ is a non-trivial element of $J_{k,m}^2$. 
\begin{proposition}
We have  
\begin{equation*}
E_{k,m}^{\text{Kl}}\left( \Phi \right) \not\in \mathbb{E}_{k,m}.
\end{equation*}
\end{proposition}
\begin{proof}
We prove the existence  of  at least one prime $p$ with $(p,m)$=1, such that
\begin{equation*}
E_{k,m}^{\text{Kl}}\left( \Phi \right) \vert_{k,m} \left(T^J(p)^\uparrow - T^J(p)^\downarrow\right) \neq 0.
\end{equation*} 
This is in particular satisfied if its  restriction to $\H^J \times \H^J$ has this property.
This  function equals
\begin{equation} \label{good}
\left( T^J(p) \otimes \text{id} - \text{id} \otimes T^J(p) \right) \Big( E_{k,m}^{\text{Kl}}
\left( \Phi \right) \Big{\vert}_{\H^J \times \H^J}\Big).
\end{equation}
In \cite{AH98} Arakawa and the second author have shown that
\begin{equation*}\label{triple}
E_{k,m}^{\text{Kl}}\left( \Phi \right) \vert_{\H^J \times \H^J} = E_{k,m}^J \otimes \Phi+  \Phi \otimes E_{k,m}^J +   G,
\end{equation*}
where 
$G \in 
\text{Sym}^2
 J_{k,m}^{\text{cusp}}$. 
 Denoting by  $\lambda_E$ and $\lambda_{\Phi}$  the eigenvalues of $E_{k,m}^J$ and $\Phi$ with
respect to the Hecke Jacobi operator $T^J(p)$ gives that   (\ref{good}) equals
\begin{equation*}
\left( \lambda_E - \lambda_{\Phi} \right)
E_{k,m}^J \otimes \Phi+  
\left(\lambda_{\Phi} - \lambda_E \right)
\Phi \otimes E_{k,m}^J +   G',
\end{equation*}
where $G' \in 
\text{Sym}^2
 J_{k,m}^{\text{cusp}}$. 
 There  exists at least one prime $p$ such that $\lambda_{E}$ is different from $\lambda_{\Phi}$,  since these eigenvalues correspond
to Eisenstein series and cusp forms of weight $2k-2$ (see \cite{SZ}).  
\end{proof} 


\begin{thebibliography}{Bri}
\addcontentsline{toc}{chapter}{Literaturverzeichnis} 
\bibitem{An87} A. N. Andrianov: \emph{Quadratic Forms and Hecke Operators,
                   Grundlehren der math. Wissenschaften.} \textbf{286}.
                     Berlin, Heidelberg,
                    New York: Springer (1987).
\bibitem{Ar94} T. Arakawa: \emph{Jacobi Eisenstein series and a basis
                     problem for Jacobi forms.} Comm. Mathematici Universitatis Sancti Pauli. 
                     \textbf{43} (1994), 181-216.
\bibitem{AH98} T. Arakawa, B. Heim: 
\emph{Real analytic Jacobi Eisenstein series and Dirichlet series attached to three Jacobi forms.} 
{\it Max-Planck Institut Bonn} \textbf{66} (1998).
\bibitem{EZ85} M. Eichler, D. Zagier:
\emph{The theory of Jacobi forms. Progress in Mathematics.} \textbf{Vol. 55}. Boston-Basel-Stuttgart: 
Birkh\"auser (1985).                     
\bibitem{Ga84}
P. Garrett: \emph{Pullbacks of Eisenstein series; applications.}
 Automorphic forms of several variables (Katata, 1983), 114--137, 
Progr. Math. \textbf{46} Birkh\"auser Boston, Boston, MA, 1984.    
\bibitem{He99} B. Heim: \emph{Pullbacks of Eisenstein series, Hecke-Jacobi theory and automorphic L-functions.}
In: Automorphic Forms, Automorphic Representations and Arithmetic. 
Proceedings of Symposia of Pure Mathematics \textbf{66}, part 2 (1999). 
 \bibitem{He06} B. Heim: \emph{On the Spezialschar of Maass.} 
{\it Preprint, submitted} 2006.
\bibitem{Ik01} T. Ikeda: \emph{On the lifting of elliptic cusp forms to Siegel cusp forms of degree $2n$.}
Ann. of Math. 
\textbf{154}   (2001), 641-681.
\bibitem{Ib} T. Ibukiyama: \emph{On Jacobi forms and Siegel modular forms of half integral weight}.
 Comment. Math. Univ. St. Pauli \textbf{41} (1992), 109-124.
\bibitem{KK05} 
W. Kohnen, H. Kojima: \emph{A Maass space in higher degree.}
Compos. Math, \textbf{141}  (2005), 313-322. 
\bibitem{Ku78} N. Kurokawa: \emph{Examples of eigenvalues of Hecke operators on Siegel cuspforms of degree two.}
                     Inventiones Math. 
                     \textbf{49} (1978), 149-165.
\bibitem{Ma79I}  H. Maass: \emph{\"Uber eine Spezialschar von Modulformen zweiten Grades I.} 
Invent. Math. \textbf{52} (1979), 95-104.
\bibitem{Ma79II} H. Maass: \emph{\"Uber eine Spezialschar von Modulformen zweiten Grades II.}
Invent. Math. \textbf{53} (1979), 249-253.
\bibitem{Ma79III} H. Maass: \emph{\"Uber eine Spezialschar von Modulformen zweiten Grades III.}
Invent. Math. \textbf{53} (1979), 255-265.                    
\bibitem{Mu89}A. Murase: \emph{L-functions attached to Jacobi forms of degree $n$. Part I. The basic indentity}
                     J. Reine Angew. Math. 
                     \textbf{401} (1989), 122-156.
\bibitem{Sh71}  G. Shimura: \emph{Introduction to the Arithmetical Theory of Automorphic Functions.} Princeton, 
Iwanami Shoten and Princeton Univ. Press, (1971).
\bibitem{SZ} N-P. Skoruppa, D. Zagier: \emph{Jacobi forms and a certain space of modular forms.} 
Invent. Math. \textbf{94} (1988), 113--146. 
\bibitem{HW} J. Hafner, L. Walling, \emph{Explicit action of Hecke operators on Siegel Modular forms}, Journal of Number Theory \textbf{93} (2002), 34-57.
\bibitem{Za80} D. Zagier:
                     \emph{Sur la conjecture de Saito-Kurokawa (d'apr\`es H. Maass}.
                     S\'em. Delange-Pisot-Poitou 1979/1980,
                     Progress in Math. \textbf{12} (1980), 371 -394. 
 \bibitem{Zi89} C. Ziegler:\emph{
       Jacobi forms of higher degree.}
       Abh. Math. Semin. Univ. Hamb. \textbf{59} (1989), 191-224.                 
 \end{thebibliography}
 \end{document}